\documentclass{amsart}
\usepackage[dvips]{color}

\usepackage{graphicx}
\usepackage{epsfig}
\usepackage{tikz}

\usepackage{latexsym}
\usepackage{amssymb}
\usepackage{amsmath}
\usepackage{amsthm}
\usepackage{amscd}
\theoremstyle{plain}
\newtheorem*{thm*}{Theorem}

\newtheorem{thm}{Theorem}
\newtheorem{lem}{Lemma}
\newtheorem{defin}{Definition}
\newtheorem{prop}{Proposition}

\theoremstyle{definition}
\newtheorem{rem}{Remark}
\newtheorem{claim}{Claim}

\newcommand{\nc}{\newcommand}
\nc\bR{\mathbb{R}}
\nc\bZ{\mathbb{Z}}

\newcommand{\calP}{{\mathcal P}}

\newcommand{\calR}{{\mathcal R}}

\newcommand{\R}{{\mathbb R}}

\DeclareMathOperator{\const}{u_k}

\author{Jon Chaika}
\author{Howard Masur}

\email{chaika@math.utah.edu}
\email{masur@math.uchicago.edu}

\address{Department of Mathematics, University of Utah, 155 S 1400 E Room 233, Salt Lake City, UT 84112}
\address{Department of Mathematics, University of Chicago, 5734 S. University Avenue, Room 208C, Chicago, IL 60637, USA}

  \thanks{J.C. partially supported by NSF grant DMS 1300550}
\thanks{H.M partially supported by NSF grant DMS 1205016}

\title[Laertes points]{There exists an interval exchange with a non-ergodic generic measure}

\begin{document}

\maketitle

\section{Introduction}
In the initial study of  interval exchange transformations in the Soviet Union,  the United States and western Europe a major motivation  was  basic ergodic theory.  One aspect was to find  bounds on the maximal number of ergodic measures.   Katok \cite{Ka} proved that $n$-IETs have at most $\lfloor \frac n 2\rfloor$ ergodic measures.  A little later Keane \cite{Keane} proved that $n$-IETs have a finite number of ergodic measures and Veech \cite{Veech78} using different methods than Katok also proved that an $n$-IET has at most $\lfloor \frac n 2 \rfloor$ ergodic measures. Various examples of non-ergodic interval exchanges were constructed by  Katok \cite{Ka}, Sataev \cite{Sataev} and  Keane \cite{Keane2}.  In a slightly different context examples were also constructed by  Veech \cite{Veech69}.  Masur \cite{Masur1} and Veech \cite{Veech} later proved
 almost every interval exchange with an irreducible permutation is uniquely ergodic.

The present paper is devoted to the study of invariant measures of minimal and not uniquely ergodic interval exchange transformations. 


Now  let $(X,d)$ be a $\sigma$- compact  metric space, $\mathcal{B}$ be the Borel $\sigma$-algebra, $\mu$ be a Borel measure and $T:X \to X$ be  $\mu$-measure preserving. 

\begin{defin}

A point $x\in X$ is said to be  generic for $\mu$ if for every continuous compactly supported $f:X \to \mathbb{R}$  we have $\sum_{i=0}^{N-1}f(T^ix) \to \int f d\mu$.
\end{defin}

Consider a $\sigma$-compact metric space. The Birkhoff ergodic  theorem and the fact that continuous compactly supported functions with supremum norm have a countable dense set says that if a measure $\mu$ is ergodic then $\mu$-almost every point is generic. 

\begin{defin}
A measure is called \emph{generic} if it is has a generic point. 
\end{defin}
The question arises if  a  non-ergodic measure can be generic.
In settings like the left shift on $\{0,1\}^{\mathbb{Z}}$ it is straightforward to build generic points for non-ergodic measures. Interval exchanges have fewer, much more constrained orbits and so these techniques do not work.
However, in this paper we show that there exists an IET with a generic non-ergodic measure.

\begin{thm}There exists a minimal non uniquely ergodic IET on $6$ intervals with $2$ ergodic measures and which has  a generic measure that is not ergodic.  
\end{thm}

\textbf{Remark:} 
An interesting question is whether there is a bound on the number of generic measures. One can use a standard Rokhlin tower argument to argue that an $n$-IET has at most $n$ measures with the property that each has a point that is generic for both $T$ and $T^{-1}$. This raises the question of whether there exists an IET, $T$, with a generic measure $\mu$ which is not generic for $T^{-1}$.   The generic point for $T$ we produce in Section \ref{sec:gp} is also generic for $T^{-1}$ by an analogous proof, so the example in this paper will not have this property. The Katok \cite{Ka}, Sataev \cite{Sataev} and Veech \cite{Veech69} examples of minimal and not uniquely ergodic IETs are such that  all of their generic measures are ergodic measures. 

\noindent \textbf{Acknowledgments:}  We thank M. Boshernitzan for asking this question and Oberwolfach where this project began.
J. Chaika thanks the University of Chicago for its hospitality.
\section{Rauzy Induction}

We follow the description of interval exchange transformations  introduced in \cite{MMY} and also explicated in \cite{AGY}.  We have an alphabet $\mathcal{A}$ on $d\geq 2$ letters.
Break an interval $I=[0,|\lambda|)$ into intervals 
$\{I_\alpha\}_{\alpha\in\mathcal{A}}$
and rearrange in a new order by translations. 
 Thus the interval
exchange transformation is entirely defined by the following data:
\begin{enumerate}
\item The lengths of the intervals
\item Their orders before and after rearranging
\end{enumerate}

The first are called length data, and are given by a vector
$\lambda\in \mathbb{R}^d$ .
The second are called combinatorial data, and are given by a pair  of bijections
$\pi=(\pi_t,\pi_b)$ from
$\mathcal{A}$ 
to
$(1,\ldots, d)$.
 The bijections 
can be viewed as a pair of rows, the top corresponding to
$\pi_t$
and the bottom corresponding to
$\pi_b$. 

 Given an interval exchange   $T$ defined by  $(\lambda,\pi)$ let $\alpha,\beta\in \mathcal{A}$ the last elements in the top and bottom.  The operation of  Rauzy induction is applied  when $\lambda_\alpha\neq \lambda_\beta$ to give a new IET  $T'$ defined by $(\lambda',\pi')$ where $\lambda',\pi'$ are as follows.   If $\lambda_\alpha>\lambda_\beta$  then $\pi'$  keeps the top row unchanged, and it changes the bottom row by moving $\beta$ to the position immediately to the right of the position occupied by $\alpha$.     We say $\alpha$ {\em  wins} and $\beta$ {\em loses}. For all $\gamma\neq \alpha$ define   $\lambda'_\gamma=\lambda_\gamma$ and define $$\lambda'_\alpha=\lambda_\alpha-\lambda_\beta.$$   

If  $\lambda_\beta>\lambda_\alpha$ then to define $\pi'$ we keep the bottom row the same and the top row is changed by moving $\alpha$ to the position to the right of the position occupied by $\beta$.  Then $\lambda'_\gamma=\lambda_\gamma$ for all $\gamma\neq\beta$ and $\lambda'_\beta=\lambda_\beta-\lambda_\alpha$.  We say $\beta$ wins and $\alpha$ loses.  

In  either case one has a  new interval exchange $T'$ determined by  $(\lambda',\pi')$
and defined on an interval $I'=[0,|\lambda'|)$ where $$|\lambda'|=\sum_{\alpha\in \mathcal{A}} \lambda'_\alpha.$$

The map $T':I'\to I'$ is the first return map 
to a subinterval of
$I$
obtained by
cutting from
$I$
a subinterval with the same right endpoint and of length
$\lambda_\zeta$
where
$\zeta$
is the loser of
the  process described above.

Let $\Delta$ be the standard simplex in $\mathbb{R}^d$ and let ${\calP}$ be the set of permuations on $n$ letters.  We can normalize so that all IET are defined on the unit interval.  
Let  $${\calR}:\Delta\times {\calP}\to \Delta\times {\calP}$$ then denote Rauzy induction.

 There is a corresponding visitation matrix  $M=M(T)$.  Let $\{e_\gamma\}_{\gamma\in \mathcal{A}}$ be the standard basis.  If $\alpha$ is the winner and $\beta$ the loser then $M(e_\gamma)=e_\gamma$ for $\gamma\neq\alpha$ and $M(e_\alpha)=e_\alpha+e_\beta$. We can view $M$ as simply arising  from the identity matrix by adding the $\alpha$ column to the $\beta$ column. 
We can projectivize the matrix $M$ and consider it as $M:\Delta\to \Delta$.

When the interval excahnge $T$ is understood, and we perform Rauzy induction $n$ times then define $M(1)=M(T)$ and inductively $$M(n)=M(n-1)M({\calR}^{n-1}T).$$
That is, the matrix $M(n)$ comes from multiplying $M(n-1)$ on the right by the matrix of Rauzy induction applied to the IET after we have done Rauzy $n-1$ times.  Let $C_\sigma(n)$ the  column of $M(n)$ corresponding to the letter $\sigma\in\mathcal{A}$. 


\begin{defin}  We say an IET satisfies the Keane condition if the orbit of every point of discontinuity is infinite and misses every other such point. It is a fact that such an IET is minimal. 
\end{defin}
If an interval exchange has all powers of Rauzy induction defined on it then it satisfies the Keane condition (see \cite[Corollary 5.4]{viana survey} for a proof in a survey).
 \begin{thm}(\cite[Section 1 and Proposition 3.22]{Veech78}  Let $T$ be an IET satisfying the Keane condition. The space of invariant measures of $T$ is $\cap_{n=1}^{\infty} M(n) \Delta$.  
  \end{thm}


We now specialize to IET on $6$ intervals.
We will   take paths that return infinitely often to the arrangement $$\begin{pmatrix}
A&B&C&D&E&F\\
F&E&D&C&B&A
\end{pmatrix}.$$ Our paths will be of two types which we call the left and right paths. Our paths will depend on a sequence of parameters $a_k$.  Inductively we assume that for the $k^{th}$ time having just traveled on the left hand side we have taken a total of $n_k$ steps of Rauzy induction and returned to the above arrangement. 
Our matrix is $M(n_k)$. 
\begin{rem} In what follows we suppress the matrix that the columns are being taken from. So $C_A(M(n_k+r_k))$ is denoted $C_A(n_k+r_k)$. Given a vector $\bar{v}$ let $|\bar{v}|$ denote the sum of the absolute values of its entries (that is its $\ell^1$ norm).
\end{rem}

Inductively assume  the columns satisfy  


\begin{equation}
\label{eq:induct1}\min\{|C_A(n_k)|,|C_B(n_k)|\}>a_k\max\{|C_E(n_k)),C_F(n_k)|\}\tag{A}
\end{equation}
 and 
\begin{equation}
\label{eq:induct2}
\min\{|C_E(n_k)|,|C_F(n_k)|\}>\frac{a_{k-1}}{k^2}\max\{|C_C(n_k)|,|C_D(n_k)|\}.\tag{B}
\end{equation}
 Now we go along the loop  on the right side.  In the first step of Rauzy induction   $F$ beats $A$. The second step    $F$  beats $B$. 
In the third step  $C$  beats $F$, and then $C$ beats $E$.  Now for a total of $r_k$ consecutive steps   $C$ loses to $D$.  The number $r_k$ is chosen so that it is smallest odd number satisfying 
  \begin{equation}
\label{CDlose}
|C_C(n_k+4+r_k)|>k |C_F(n_k+4+r_k)| \tag{C}
\end{equation}

Now $C$ will beat $D$. 
 Now $F$  beats $C$ and then $F$  beats $D$. Next $E$ beats $F$  consecutively $s_k$ times  for a total of $s_k+r_k+7$ steps until the first time  that 
 \begin{multline}
\label{eq:FbeatsA} |C_F(n_k+r_k+s_k+7)|> a_{k+1} \tag{D}
 \cdot
 \max\{|C_A(n_k+7+r_k+s_k)|,|C_B(n_k+7+r_k+s_k)|\}.
 \end{multline}
  Then $F$ beats $E$ and we return to $$\begin{pmatrix}
A&B&C&D&E&F\\
F&E&D&C&B&A
\end{pmatrix}$$ for the $k+1^{st}$ time. 
The number of steps we have taken is $r_k+s_k+8$ so 
 $$n_{k+1}=n_k+r_k+s_k+8.$$ 
Notice when we return to $\begin{pmatrix}
A&B&C&D&E&F\\
F&E&D&C&B&A
\end{pmatrix}$
we have  by (\ref{eq:induct1})
 $$\min\{|C_E(n_{k+1})|,|C_F(n_{k+1})|\}>a_{k+1}\max\{|C_A(n_{k+1})|,|C_B(n_{k+1})|\}$$ 
and for $a_k$ sufficiently large compared to $k$,  
   $$\min\{|C_A(n_{k+1})|,|C_B(n_{k+1})|\}>\frac{a_k}{k^2}\max\{|C_C(n_{k+1})|,|C_D(n_{k+1}|)\}$$ so (\ref{eq:induct2}) holds.  
   
The diagrams below describe the  paths on the right side and left side of the Rauzy diagram. The permutation$\begin{pmatrix} A&B &C&D&E&F\\ F&E & D&C&B&A \end{pmatrix}$ is a common point.

\begin{figure}[h]\label{fig:rauzy-diagram}
\begin{tikzpicture}
\path[->] node{$\begin{pmatrix} A&B &C&D&E&F\\ F&E & D&C&B&A \end{pmatrix}$} (2,0) edge (3,0);
\path[->] (5,0) node{$\begin{pmatrix} A&B &C&D&E&F\\ F&A&E & D&C&B \end{pmatrix}$};
\path[->] (5,-.5) edge (5,-1) (5,-1.5) node {$\begin{pmatrix} A&B &C&D&E&F\\ F&B&A&E & D&C \end{pmatrix}$};
\path[dotted, ->] (7.2,-1.5) edge (8,-1.5) (10,-1.5) node{$\begin{pmatrix} A&B &C&F&D&E\\ F&B&A&E&D&C \end{pmatrix}$};
\path[dotted,->](9,-2) edge (9,-2.5) (10,-3) node {$\begin{pmatrix} A&B &C&E&F&D\\ F&B&A & E&D&C \end{pmatrix}$};
\path[dotted,->] (8,-3) edge (6.8,-2);
\path[->](12,-3) edge[loop right] (13,-3);
\path[->] (12.6,-2.6) node{$r_k$ times};
\path[->] (5,-2) edge (5,-2.5) (5,-3) node {$\begin{pmatrix} A&B &C&D&E&F\\ F&C&B&A&E & D \end{pmatrix}$};
\path[->] (5,-3.5) edge (5,-4) (5,-4.5) node{$\begin{pmatrix} A&B &C&D&E&F\\ F&D&C&B&A&E \end{pmatrix}$};
\path[->] (2.8,-4.4) edge (0,-.7);
\path[dotted, ->] (7,-4.4) edge[loop right] (7.5,-4.4) (7.9,-4.2) node{$s_k$ times};


\end{tikzpicture}
\caption{The path we take on the right hand side of the Rauzy diagram}
\end{figure}
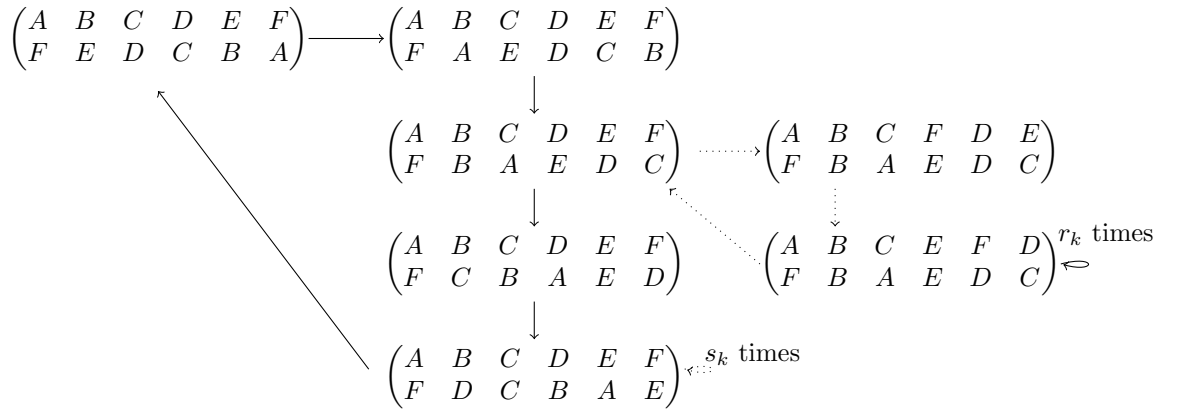

\begin{figure}[h]\label{fig:rauzy-diagramLHS}
\begin{tikzpicture}
\path[->] node{$\begin{pmatrix} A&B &C&D&E&F\\ F&E & D&C&B&A \end{pmatrix}$} (-2,0) edge (-3,0);
\path[->] (-5,0) node{$\begin{pmatrix} A&F&B &C&D&E\\ F&E & D&C&B&A  \end{pmatrix}$};
\path[->] (-5,-.5) edge (-5,-1) (-5,-1.5) node {$\begin{pmatrix} A&E&F&B &C&D\\ F&E & D&C&B&A \end{pmatrix}$};
\path[dotted, ->] (-7.2,-1.5) edge (-8,-1.5) (-10,-1.5) node{$\begin{pmatrix} A&E&F&B&C&D\\ F&E&D&A&C&B\end{pmatrix}$};
\path[dotted,->](-9,-2) edge (-9,-2.5) (-10,-3) node {$\begin{pmatrix} A&E&F&B &C&D\\ F&E & D&B&A&C \end{pmatrix}$};
\path[dotted,->] (-8,-3) edge (-6.8,-2);
\path[->](-12,-3) edge[loop left] (-13.3,-3);
\path[->] (-12.7,-2.6) node{$r_k$ times};
\path[->] (-5,-2) edge (-5,-2.5) (-5,-3) node {$\begin{pmatrix} A&D&E&F&B&C\\ F&E & D&C&B&A  \end{pmatrix}$};
\path[->] (-5,-3.5) edge (-5,-4) (-5,-4.5) node{$\begin{pmatrix} A&C&D&E&F&B\\ F&E & D&C&B&A \end{pmatrix}$};
\path[->] (-2.8,-4.4) edge (0,-.7);
\path[dotted, ->] (-7.3,-4.4) edge[loop left] (-7.5,-4.4) (-7.9,-4.2) node{$s_k$ times};


\end{tikzpicture}
\caption{The path we take on the left hand side of the Rauzy diagram}
\end{figure}

\pagebreak
  On the left hand side $A$ beats $F$ and then beats $E$; then $D$ beats $A$ and then  $B$; then $C$ beats $D$ repeatedly  $r_{k+1}$ times until  the first time that  
$$C_D>(k+1)C_A.$$
Then $D$ beats $C$.   
Then $A$ beats $D$ and then $C$ and then $B$ beats $A$ repeatedly 
$s_{k+1}$ times  until   the first time that  
\begin{multline*}C_A(n_{k+1}+r_{k+1}+7+s_{k+1})>a_{k+2}
 \min\{(C_E(n_{k+1}+7+s_{k+1}+m_{k+1}),C_F(n_{k+1}+7+r_{k+1}+s_{k+1})\}.
 \end{multline*}
Finally $A$ beats $B$ and we return.  

Our main Theorem is then the following.  
\begin{thm}
For choices of $a_k$ there is a minimal  $6$ interval exchange transformation   $T$ with  the  property  that the resulting simplices $\Delta_k=M(k)(\Delta)$ converge to a line segment as $k\to\infty$.  The endpoints are ergodic measures for $T$.    One endpoint is the common projective limit of the column  vectors $C_A,C_B$; the other endpoint  is the common projective limit of the column vectors $C_E,C_F$.  Along the sequence $n_k$ the column vectors corresponding to $C_C,C_D$ converge projectively to a common  interior point of the segment and this limit point is a generic measure.
\end{thm}


  
 \subsection{Outline of proof}

The proof is in two main steps.  In the next subsection we show that   $\cap_{n=1}^{\infty}M(n)\Delta$ converges to a line segment. This is given by Proposition~\ref{prop:conv} with preliminary Proposition~\ref{prop:vectors}.   This shows that there are two ergodic measures. 
  We also show that the sequence of vectors given by the  C$^{th}$ and D$^{th}$ columns of the matrices $M(n)$ converge to a single point away from the endpoints in $\cap_{n=1}^{\infty}M(n)\Delta$.
  This gives the  invariant measure $\nu$ we are interested in.  In the section after that we show that $\nu$ is generic by finding a generic point $x$.

  \subsection{Convergence of column vectors} 
 
The first Lemma is simple Euclidean geometry. Our angles will be bounded away from $\frac \pi 2$ and so $\sin x$ will be comparable with $x$.
\begin{lem} 
\label{lem:angles}Let $v, w \in \R^n,$ and let $\theta(v,w)$ denote the angle between $v$ and $w$. If $\theta(v+w,w)$ denotes the angle between $v+w$ and $w$, we have $$| \sin \theta(v+w,w) |= \frac{\|v\|}{\|v+w\|} |\sin \theta(v,w)|.$$  \end{lem}

We will now simplify notation for the column  vectors, using $v$ for a vector and eliminating some of the times of Rauzy induction.

We assume that we  enter  the right side after an even number of returns.  
We define   $v_F(2k+1),v_E(2k+1)$ to be the column vectors  $C_F(n_{2k}+3), C_E(n_{2k}+4)$ after  $F,E$ lose to $C$ on the right side.  We choose  $v_F(2k+2)$  to be the column vector $C_F(n_{2k}+r_{2k}+m_{2k}+7)$ after  $E$ is added to $F$ on the right side and $v_E(2k+2)$ the column vector $C_E(n_{2k}+r_k+8)=C_E(n_{2k+1})$ after $C_F$ is added to $C_E$.

Similarly  let $v_A(2k)$ be the column vector  $C_A(n_{2k-1}+3)$  on the  left side just after $A$   has lost to $D$ and $v_A(2k+1)$ is the column vector  $C_A(n_{2k})=C_A(n_{2k-1}+r_{2k-1}+m_{2k-1}+7)$ just after $B$ has finished beating $A$. 
Similarly the column vectors $v_B(2k)$ and $v_B(2k+1)$ are defined in terms of  $C_B$.

We define $v^R_C(2k+1)$  to be the column vector $C_C(n_{2k}+r_{2k}+4)$ when $C$ has finished losing to $D$ on the right side and $v^R_C(2k+2)$  the column vector for $C$  after  $C$ loses to $F$.  This ensures that $v^R_C(2k+2)$ lies on the line joining $v^R_C(2k+1)$ and $v_F(2k+1)$ which will be needed later.   Similarly, $z^L_C(2k)$ is the column vector  for $C$ on the left side after $D$ has beaten $C$, and $z^L_C(2k+1)$ is the column vector after $A$ has beaten $C$. We have similar definitions for $v_D^L(2k),v_D^L(2k+1)$ in terms of the column vectors $C_D$.

The following Proposition gives the properties these vectors satisfy. They will be used to prove the desired convergence results.
  
\begin{prop}
\label{prop:vectors}
With  a  choice of  $a_k$ there exist  constants $\delta>0$, $\alpha<1$  
such that  for $i,j\in\{1,2\}$, 
and all $k$
\begin{enumerate}
\item for $\sigma,\sigma'\in\{E,F\}$,  $\theta(v_\sigma(k),v_{\sigma'}(k+1))<\alpha^k$ and similarly for $A,B$.
\item for $\sigma\in\{E,F\}$ and $\sigma'\in\{A,B\}$,   $\theta(v_\sigma(k),v_{\sigma'}(k))\geq\delta$. 
\item $\theta(v^R_C(2k+2),v_F(2k+1))=(1-\frac 1 k+O(\frac 1 {k^2}))\theta(v^R_C(2k+1),v_F(2k+1))$, 
with the same formula for $v^R_D$.
\item $\theta(v^L_D(2k+1),v_A(2k))=(1-\frac 1 k+O(\frac 1 {k^2}))\theta(v^L_D(2k),v_A(2k))$, with the same formula for $v^L_C$. 
\item  $\theta(v_C^{R}(k),v_D^{R}(k))\leq\alpha^k$ for all $k$. The same holds on the left.
\item $\theta(v_C^L(2k),v^R_C(2k))\leq\alpha^k$.
\item $\theta(v^R_D(2k+1),v_D^L(2k+1))\leq\alpha^k$.
\end{enumerate}

Moreover $v^R_C(2k+2)$ lies on the line connecting $v_C^R(2k+1)$ and $v_F(2k+1)$ and $v_C^L(2k+1)$  lies on the line connecting $v_C^L(2k)$ and $v_A(2k)$ with the same  for $v_D$. 
  
\end{prop}

\begin{proof}
The vectors $v_E(2k+1)$ (resp. $v_F(2k+1)$) arises from  $v_E(2k)$ (resp.$v_F(2k)$) from $E$ (resp. $F$)   losing first to  $A$ on the left side and then to $C$ back on the right side. Recall this means that the $A$ column and then the $C$ column are added to the $E$ (resp.$F$ column).   Then by   
(\ref{eq:induct1}) and (\ref{eq:induct2}) and  Lemma~\ref{lem:angles}, 
  $$\theta(v_E(2k+1),v_E(2k))\leq \frac{1}{a_{k+1}+\frac{a_{k+1}}{k^2}},$$ and similarly for $v_F$.  
Similarly we have  $$\theta(v_A(2k+2),v_A(2k+1))\leq \frac{1}{a_{k+2}+\frac{a_{k+2}}{(k+1)^2}}$$ and the same for $v_B$.    


On the right side the   $C_E$ column is added  to the $C_F$ column  many times and then    $C_F$ is added to $C_E$.  By (\ref{eq:induct1}) and (\ref{eq:FbeatsA})  and Lemma~\ref{lem:angles}, this gives $$\theta(v_F(2k+2), v_E(2k+2))\leq \frac{1}{a_ka_{k+1}}\theta(v_F(2k+1),v_E(2k+1)).$$

The same argument shows that
 $$|\theta(v_F(2k+2),v_F(2k+1))-\theta(v_F(2k+1),v_E(2k+1))|\leq \frac{1}{a_ka_{k+1}}.$$ 
We have similar inequalities for the $v_A,v_B$.  
 If we let $a_k$ be exponentially small,  putting these together we see  (1) is  satisfied.

We also note that  on the right side $F$ beats $A,B$ only when smaller by the factor of $a_k$ and similarly on the left side $A$ beats $E,F$ only when smaller by a factor of $a_{k+1}$ which gives  $$\theta(v_E(k+1),v_A(k+1))\geq (1-(\frac{1}{a_k}+\frac{1}{a_{k+1}}))\theta(v_E(k),v_A(k)).$$
Thus since initially the angle between the column for $E$ and $A$ is positive  we can choose $a_k$  so  (2) is satisfied. 

We prove (3),(4). When $C$ or $D$ loses to $F$,
$C_C$ moves towards $C_F$ so
$$\theta(v_C^R(2k+2),v_C^R(2k+1))\leq \frac{1}{k}\theta(v_C^R(2k+1),v_F(2k)).$$
The same holds for $v_C^L(2k), v_C^L(2k+1)$ and also for the corresponding $v_D$.  
These imply (3),(4). 

When we have finished on the right we have $v_C^R(2k),v_D^R(2k)$. When we move to the left  we change from  $v_C^R(2k),v_D^R(2k)$ to  $v_C^L(2k),v_D^L(2k)$.
In doing  so we add the $C$ column  to the $D$ column $r_{2k}$ times and then the $D$ column is added back to the $C$ column.
Since $r_k$ grows exponentially this give (5) for $j$ even and left.  Moving from left to right similarly gives  (5) when $j$ is odd and we are on the right.  In addition  in increasing the index on both sides means adding the same column (either $C_A$ or $C_F$) to each so they remain exponentially close.  This proves (5).  
Similarly the fact that $C$ beats $D$ repeatedly on the left  means $v_D^L(2k)$ moves close to $v_C^R(2k)$ and then the fact that $D$ is added to $C$ means that  $v_C^L(2k)$ is close to $v_C^R(2k)$. This is (6). We have the corresponding computation on  the right giving (7).

\end{proof}

The following is the basic Proposition  that will allow us to find limiting invariant measures.

\begin{prop}
 \label{prop:conv}

Consider the sequence of vectors  $v_\sigma(k)\in \mathbb{R}^6$ that satisfy the conclusions of Proposition~\ref{prop:vectors}, projectivized  to lie in the simplex  $$\Delta=\{x_1,\ldots, x_6:\sum_{i=1}^6 x_i=1\}.$$ 

Then 
\begin{itemize}
\item $v_E(k),v_F(k)$ have limit $v_0$ and $v_A(k),v_B(k)$ have a distinct limit $v_1$.
\item $\underset{k \to \infty}{\lim}\, v_C^L(k)=\lim v_C^L(k)=\underset{k \to \infty}{\lim}\, v_D^L(k)=\lim v_D^R(k)=v'$ exists in $\mathbb{P}(\mathbb{R}^6)$
\item $v'$ is on the line between $v_0$ and $v_1$ in  $\mathbb{P}(\mathbb{R}^6)$
\item $v'\neq v_0,v_1$.
\end{itemize}
\end{prop}

\begin{proof} 
 
We can replace angles with Euclidean distance.  
The fact that the limits $v_0,v_1$ of $v_{E,F}(k), v_{A,B}(k)$ exist 
 follows immediately   from (1). In fact, they converge exponentially fast.  The fact that the limits are distinct comes from  (2).

Let $\bar{u}$ be the direction of the line connecting $v_0$ and $v_1$. 
Let $y$ denote a coordinate perpendicular to $\bar{u}$.  We claim that the $y$ coordinates of $v_C^R(k),v^R_D(k),v_C^L(k),v_D^L(k)$ go to $0$.  
Without loss of generality consider $v_C^R(k),v_C^L(k)$. The fact that  $v_C^R(2k+1)$ and  $v_C^R(2k+2)$ lie on a line through $v_F(2k+1)$  whose  $y$ coordinate goes  to $0$ exponentially fast   together with (3)  implies that the  $y$ coordinate of $v_C^R(2k+2)$  must decrease by a factor of approximately $1-\frac 1 k+O(\frac{1}{k^2})$ from the coordinate of $v_C^R(2k+1)$.   The same is true for the change from even to odd indices of  $v_C^L$, also  by (3).  
But then (6) says that  $d(v_C^R(2k),v^L_C(2k))\to 0$ exponentially fast  and together with the fact that $$\prod_{k=2}^\infty(1-\frac 1 k)=0$$ implies the $y$ coordinates of $v_C^{R,L}(k)$ must converge to $0$. 
The same is true of $v_D^{R,L}(k)$ by (3) and (7).

It follows from the above conditions that these $v_C^R(k), v_C^L(k),v_D^R(k),v_D^L(k)$ have a single limit $v'$ on the axis. Indeed  since the $\bar u$ coordinate of $v_C^R(2k+2)$ decreases by  $\frac{1}{2k+1}$ with small error from that of $v_C^R(2k+1)$ and then increases    by $\frac{1}{2k+2}$ with small error from even to odd implies that the distance in the $\bar{u}$ coordinate between $v_C^R(2k)$ and $v_C^R(2k+2)$ is $O(\frac 1 {k^2})$ so that these  coordinates form  a convergent Cauchy sequence.  The fact that the   
distance in the $\bar{u}$ coordinate between $v_C^R(2k)$ and $v_C^R(2k+1)$ goes to  $0$ means the entire sequence converges. The fact that (5) holds implies that the $v_D^{R,L}(k)$ have the same limit.

We claim that $v'\notin \{v_0,v_1\}.$   Suppose by contradiction  that $v'=v_0$. (The argument for $v_1$ is similar)   Assume without loss of generality $|v_1-v_0|=1$.   For $k$ even   $$d(v_D^L(k+1),v_D^L(k))=(1+o(1))  \frac 1 k(|1-v_D^L(k)|).$$ That is, since $v_D^L(k+1)$ is assumed near $0$,  we move towards $v_1$ an amount about $\frac 1 k$.   When we move towards $0$ on the next move we move not as far;
$$d(v_D^R(k+2),v_D^R(k+1))=(1+o(1)) \frac{1}{k+1}v_D^R(k+1).$$  Then again on the next move we move further away.     On odd moves we will move further away from $0$ than we move closer to $0$ on even moves as soon as say the $\bar{u}$  coordinate is at most  some positive $\epsilon$.  This says we do not converge to $0$.
 By a symmetric argument we see that no subsequence converges to $q$.  
\end{proof}

  \section{Getting a generic point}\label{sec:gp}
  \begin{lem} Let $\mu$ be a $T$-invariant measure with intervals $I_j$.  A point $x$ is generic for $\mu$ if and only if
  $\underset{n \to \infty}{\lim} \, \frac 1 n \sum_{k=0}^{n-1}\chi_{I_j}(T^k(x))=\mu(I_j)$ for all $j$.
  \end{lem}
\begin{proof}  The only if direction holds by definition.  For the other direction, the point $x$ defines an invariant measure $\nu$ such that 
$$\nu(I_j)=\underset{n \to \infty}{\lim} \, \frac 1 n \sum_{k=0}^{n-1}\chi_{I_j}(T^k(x)).$$  We conclude that   $\nu(I_j)=\mu(I_j)$. However the measures of initial intervals determine the invariant measure. This is because if $T$ is minimal then $\{T^iI_\sigma\}_{i \in \mathbb{Z} ,\sigma \in \{A,...,F\}}$ generates the Borel $\sigma$-algebra.
  \end{proof}
  
Given $x \in [0,1)$, for each $j$  let $$v_j(x)=(\sum_{i=0}^{j-1}\chi_A(T^ix),...,\sum_{i=0}^{j-1}\chi_F(T^ix))$$ the vector of visits to the original intervals.  
If $T^i$ is continuous on an interval $U$ for $0\leq i<j$ then $v_j$ is constant on $U$ and in an abuse of notation let $v_j(U)$ be this vector.

 Recall $n_k$ is the number so that after $n_k$ steps of Rauzy induction we have   hit 
$$\begin{pmatrix}
A&B&C&D&E&F\\
F&E&D&C&B&A
\end{pmatrix}.$$

 Let
$J_k$ be the interval so that after $n_k$ steps of Rauzy induction the induced map, denoted $S_k$,  is defined  on $J_k$.
Let $A_k,...,F_k$ be the six subintervals of the IET $S_k$. 
For  $\sigma_k$  one  of these intervals and $x\in \sigma_k$ 
$$m_k(x)=\min\{i>0: T^ix\in J_k\}$$ is constant on $\sigma_k$ and is $|C_{\sigma_k}|$. 


Let $m_k(\sigma_k)$ be this number.  In other words in this notation $$v_{m_k(\sigma_k)}(\sigma_k)=C_{\sigma_k}.$$ Also let  $${\mathcal{O}_k(\sigma_k)=\cup_{i=0}^{m_k(\sigma_k)-1}T^i(\sigma_k)}.$$
This is the orbit of $\sigma_k$ under $T$ before $\sigma_k$ returns to $J_k$.

Now suppose we have the induced map $S_k$ on $J_k$ and we are about to go to the left side of Rauzy induction.   
Recall 
 $r_{k+1}$ is  the number of times  on the left side that $C$ beats $D$, and  so for $0\leq i\leq r_{k+1}$, $${S_k^i(C_k) \cap C_k\neq \emptyset}.$$  If on the other hand we are about to go to the  right side  then for $0\leq i\leq r_{k+1}$ $$S_k^i(D_k) \cap D_k\neq \emptyset.$$  
For simplicity let us assume that we are about to go to the left hand side. The arguments are symmetric if we are about to go to the right hand side.  Recall $r_{k+1}$ is odd.  Let $$\const:=\lfloor \frac{r_{k+1}}{2}\rfloor.$$

 Let
$$\hat{I}_k=S_k^{\const}(C_k)\cap S_k^{-\const}(C_k)\cap C_k\subset J_k.$$

\vspace{5mm}

\begin{lem}
$\hat{I}_k$ consists of two  intervals of
$\mathcal{O}_{k+1}(C_{k+1})$ and one interval  of $\mathcal{O}_{k+1}(D_{k+1})$
\end{lem}
\begin{proof}

On the left side  $D_k$  beats $A_k$ and $B_k$.   Then $C_k$ as the bottom letter   beats $D_k$, the top letter $r_{k+1}$ times before losing to it.  Set $$D'_k=D_k\setminus (S_k(A_k)\cup S_k(B_k)).$$  For each $i$,   after $C_k$ has beaten $D_k'$   $i$ times,  we have $S_k^{-i}(D_k')\subset C_k$.  Thus 
\begin{multline}
C_k=S_k^{-1}(D'_k)\cup S_k^{-2}(D'_k) \cup S_k^{-3}(D'_k)... \\ \cup S_k^{-r_{k+1}}(D'_k)\cup( S_k^{-r_{k+1}-1}(D'_k)\cap C_k),
\end{multline} with this disjoint union moving right to left along $C_k$.  The last term in the union which is at the left end of $C_k$ is an image of $C_{k+1}$.  Moreover   $D'_k$ is a disjoint union of an image of $C_{k+1}$ and an image of $D_{k+1}$.  In addition $S_k$ acts by translation by  the  length of $D'_k$  along $C_k$.  Then  for $i<r_{k+1}$, $S^i_k(C_k)\cap C_k$ is $C_k$ with  the $i$ copies of $C_{k+1}\cup D_{k+1}$ furthest on the left of $C_k$ deleted.   Similarly, $S^{-i}_k(C_k)\cap C_k$ deletes the $i$ copies of $C_{k+1}\cup D_{k+1}$ furthest on the right of $C_k$. It follows then from the  definition of $\hat{I}_k$ that it is a single union of  an image of  $C_{k+1}\cup D_{k+1}$ and the last term in the union, which is an image of $C_{k+1}$. 
\end{proof}

\begin{rem}

If the top letter $\sigma$  beats the bottom letter $\tau$ then $S_k(\tau)\subset \sigma$.  Similarly
if the bottom letter  $\tau$ beats the top letter $\sigma$  then $S^{-1}(\sigma)\subset \tau$.
With this observation, the proof of the
previous lemma on the right hand side would change the negative powers of
$S_k$ to positive powers and the positive powers of $S_k$ to negative powers.
\end{rem}



Now  let $I_k=\hat{I}_k\cap \mathcal{O}_{k+1}(D_{k+1})$.
We conclude from the previous lemma that


\begin{lem}
$S_k^iI_k\subset C_k$ for $-\const<i<\const.$
\end{lem}
 Let $$\hat{m}_k=\min\{i \geq 0:T^iI_k\subset J_{k+1}\}.$$ 

Then $v_{\hat{m}_k}(I_k)$ consists of  $\const$ iterates of the entire orbit $\mathcal{O}_k(C_k)$, followed by    the entire  orbit of $\mathcal{O}(D_k)$, and  then followed by the entire orbit of $\mathcal{O}(A_k)$.
Now the fact  that 
 $$\const m_k(C_k)=(1+o(1)) \frac k 2 m_k(A_k)$$ implies 
$$m_k(A_k)=o(\hat{m}_k).$$  Let $\mathcal{I}_{k-1}$ be an interval that has been inductively defined so that it contains a unique interval of $\mathcal{O}_k(C_k)$. Let $\mathcal{I}_k$ be defined as $\mathcal{O}_k(I_k)\cap \mathcal{I}_{k-1}$. Observe that it contains a unique interval of $\mathcal{O}_{k+1}(D_{k+1})$ and therefore the
 number $$j_k=\min\{j:T^j\mathcal{I}_k\subset J_{k+1}\}$$
is well-defined and $T^{j_k}$ is continuous on $\mathcal{I}_k$.

\begin{lem} 
\label{lem:v} 
$\theta(v_{m_k(C_k)}(C_k), v_{j_k}(\mathcal{I}_k))=o(1)$ and
 $\theta(v_{m_k(D_k)}(D_k), v_{j_k}(\mathcal{I}_k))=o(1)$.
\end{lem}
In fact these angles approach $0$ exponentially fast.

\begin{proof}
This is similar to the proof  of Proposition~\ref{prop:vectors}  because  $v_{j_k}(\mathcal{I}_k)$ is made up of blocks of $C_{k}$ or $D_{k}$ and a smaller column of  $\mathcal{O}(A_k)$.
\end{proof}


We conclude the proof of the Main Theorem.  We let 
$$x=\cap_{k=1}^\infty \mathcal{I}_k.$$ We will show that $x$ is generic for $\nu$.   To do  that we show that the angle its visitation vector makes with the column $C_k$ goes to $0$ as $k\to\infty$. 
This is  shown in the next Proposition.

\begin{prop}
\label{prop:finish}For $j\in [j_{k-1},j_k]$ we have that $$\lim_{k\to\infty}\theta(v_{j}(\mathcal{I}_{k}),v_{m_{k}(C_k)}(C_k))=0.$$ 
\end{prop}
\begin{proof}
By Lemma~\ref{lem:v}, the angle $v_{j_{k-1}}(\mathcal{I}_k)=v_{j_{k-1}}(\mathcal{I}_{k-1})$ makes with $v_{m_{k-1}(D_{k-1})}(D_{k-1})$ goes to $0$ with $k$ and by 
Proposition~\ref{prop:vectors}  the angle of the latter with  $v_{m_k(C_k)}(C_k)$ also goes to $0$.   Thus the conclusion of the Proposition holds for  $j=j_{k-1}$.
 Moreover $$T^{j_{k-1}}(\mathcal{I}_k)=S_k(I_k) \subset C_k.$$

We wish to  increase the set of  times in the orbit of $\mathcal{I}_k$.  
Again recall that  we are assuming that we have just finished the right hand side.  Since $$j_k=j_{k-1}+(\const-1)m_k(C_k)+m_k(A_k),$$ we first consider $m_k(C_k)$. Now  $v_{m_k(C_k)}(C_k)$ consists of  $r_k$ copies of the vector of visits of the entire orbit $\mathcal{O}_{k-1}(D_{k-1})$ to the original intervals  followed by a single copy of the vector of visits of $\mathcal{O}_{k-1}(C_{k-1})$ to the original intervals followed by a single copy of $\mathcal{O}_{k-1}(F_{k-1})$. 
Given that  $m_{k-1}(D_{k-1})=o(j_{k-1})$,   Proposition~\ref{prop:finish} holds for any $i$ in the initial  interval $[j_{k-1}, j_{k-1}+m_{k-1}(D_{k-1})]$.  Then by Proposition~\ref{prop:vectors}  and the fact that the Proposition~\ref{prop:finish} holds  for $j_{k-1}$, 
we have that $$\theta(v_{j_{k-1}}(\mathcal{I}),v_{m_{k-1}(D_{k-1})}(D_{k-1}))=o(1).$$ This implies that for any   
$$i \in [j_{k-1},j_{k-1}+r_km_{k-1}(D_{k-1})],$$  if one considers the orbit of $\mathcal{I}_k$ at time $i$, 
 then $$\lim_{k\to\infty}\theta(v_i(\mathcal{I}_k), v_{m_k(C_k)}(C_k))=0.$$ Since  $m_{k-1}(F_{k-1})=o(r_km_{k-1}(D_{k-1}))$ we have that if $i \in [j_{k-1},j_{k-1}+m_{k}(C_k)]$ then $$\lim_{k\to\infty}\theta(v_i(\mathcal{I}_k),v_{m_k(C_k)}(C_k))=0.$$ 
Iterating this shows that if 
$${i\in [j_{k-1},j_{k-1}+(\const-1)m_k(C_k)]}$$
 we have that $$\lim_{k\to\infty}\theta(v_i(\mathcal{I}_k),v_{m_k(C_k)}(C_k))=0.$$ 
 Finally since  $m_k(A_k)=o(\const m_k(C_k))$ and 
 $$j_k=j_{k-1}+(\const-1)m_k(C_k)+m_k(A_k),$$
  this finishes the proof  of the Proposition and the proof of the main theorem is complete.
\end{proof}

\end{document}